\documentclass{IEEEtran}
\usepackage{amssymb,amsthm,amsmath}
\usepackage{graphicx}
\usepackage{xspace}
\usepackage{tikz}
\usepackage[bookmarks=true,hyperindex=true]{hyperref}
\hypersetup{
    pdftitle = {A short proof of the Gacs-Korner theorem},
    pdfauthor = {Laszlo Csirmaz},
    bookmarksopen = false,
    pageanchor = false,
    plainpages = false,
}
\makeatletter
\def\correcthref{\hyper@anchor{\@currentHref}}
\makeatother
\newtheorem{theorem}{Theorem}
\newtheorem{lemma}[theorem]{Lemma}

\newcommand\eps{\varepsilon}
\newcommand\T{\mathcal{T}}
\newcommand\R{\mathbb{R}}

\DeclareMathOperator\GK{C_{\mathsf{GK}}}

\newcommand\ing{\begin{tikzpicture}[scale=0.13\ht\strutbox]%
        \draw[line width=0.015cm](0,0)--(0.055,0.22)--(0.22,0.22);
        \draw[line width=0.035cm](0.22,0.22)--(0.165,0);
        \draw[line width=0.025cm](0.165,0)--(0,0);
   \end{tikzpicture}}
\renewcommand\Delta{\begin{tikzpicture}[scale=0.13\ht\strutbox]%
       \draw[line width=0.015cm](0,0)--(0.137+0.02,0.22);
       \draw[line width=0.035cm](0.137+0.02,0.22)--(0.165+0.04,0);
       \draw[line width=0.025cm](0.165+0.04,0)--(0,0);
   \end{tikzpicture}}
\newcommand\XYZ{{XY\!Z}}
\newcommand\I{I\xspace}
\renewcommand\H{H\xspace}
\def\|{\mkern 1.4mu\mathord{|}\mkern 1.4mu}
\renewcommand{\ge}{\geqslant}
\renewcommand{\le}{\leqslant}

\title{\huge\bf A short proof of the G\'acs--K\"orner theorem}
\author{\IEEEauthorblockN{Laszlo Csirmaz}\\
\IEEEauthorblockA{R\'enyi Institute, Budapest and UTIA,
Prague}\\
\IEEEauthorblockA{email: csirmaz@renyi.hu}}

\begin{document}
\maketitle
\begin{abstract}
We present a short proof of a celebrated result of G\'acs and K\"orner
giving sufficient and necessary condition on the joint distribution of two
discrete random variables $X$ and $Y$ for the case when their mutual information
matches the extractable (in the limit) common information. Our proof is
based on the observation that the mere existence of certain random variables
jointly distributed with $X$ and $Y$ can impose restriction on all random
variables jointly distributed with $X$ and $Y$.
\end{abstract}


\section{Preliminaries}\label{sec:intro}

Ahlswede, G\'acs, K\"orner, Wyner \cite{a-k,a-2,gacs-korner,wyner} studied
the problem of extracting common information from a pair of correlated
discrete memoryless sources $X$ and $Y$. The G\'acs-K\"orner common
information, denoted by $\GK=\GK(X;Y)$, is the rate of randomness which can
be simultaneously extracted from both sources. This rate is clearly bounded
by the mutual information $\I(X;Y)$. The celebrated result of G\'acs and
K\"orner says that $\GK$ is zero unless the non-zero elements of the joint
probability matrix have a block structure, and achieves the theoretical
maximum only when within each block the sources are independent. Their proof
in \cite{gacs-korner} is quite involved and uses deep theorems on ergodic
processes. Based on Romachchenko's idea of chain independence
\cite{Rom2000}, in \cite{MMchain} Makarychev et al presented a direct, while
still lengthy and technical, proof of the first part of G\'acs and
K\"orner's result. In this note we give a short proof of both parts.

Our starting point is the tension region $\T(X;Y)$ introduced by Prabhakaran
et al in \cite{PP14}. This region is the set of the triplets of real numbers
$$
\langle\: \I(X;Z\|Y),\: \I(Y;Z\|X),\:\I(X;Y\|Z)\:\rangle \in \R^3,
$$
where $Z$ runs over all random variables jointly distributed with $X$ and
$Y$. Observe that this definition differs slightly from the one in
\cite{PP14} as the increasing hull operator is not applied to this point
set. Figure \ref{fig:ten1} illustrates the lower part of this region 
facing the origin for a particular distribution. 
More information about this region can be found in \cite{Li-Gamal}.
\begin{figure}
\centering
\includegraphics[width=6cm]{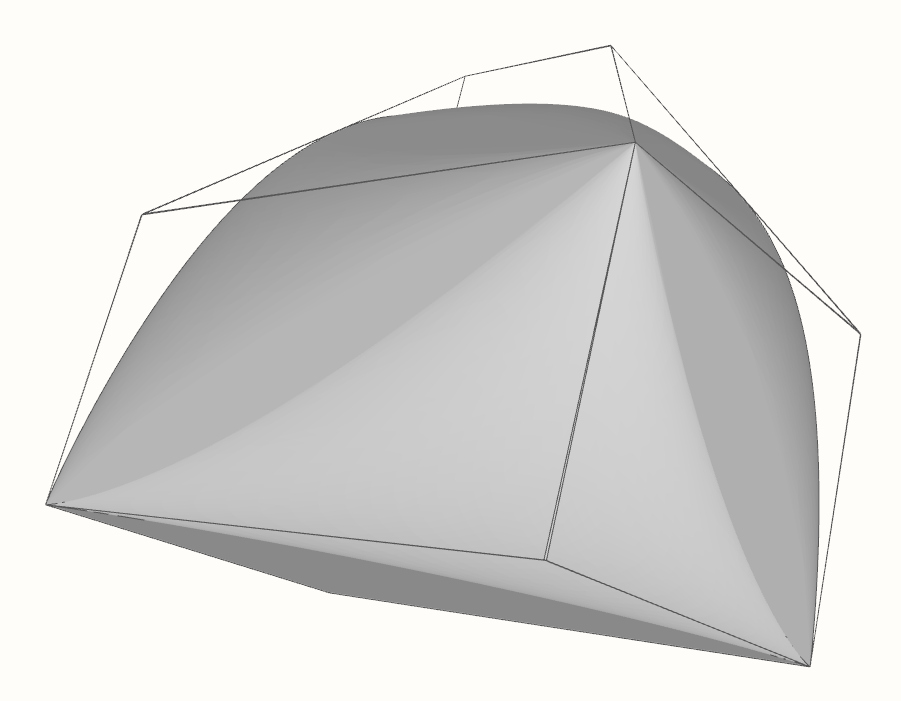}
\begin{tikzpicture}\useasboundingbox(0,0) rectangle (0,0);
\draw[->] (-2.49, 0.93) -- (-1.66,4.8);
\draw (-1.66,4.8) node[below right] {\footnotesize$\I(X;Y|Z)$};
\draw[->] (-2.49, 0.93) -- (-6.5, 1.38);
\draw (-6.5,1.28) node[below] {\footnotesize$\I(X;Z|Y)$};
\draw[->] (-2.49, 0.93) -- (-0.19, 0.01);
\draw (0.1,0.1) node[above] {\footnotesize$\I(Y;Z|X)$};
\end{tikzpicture}
\caption{Tension region of a binary distribution}\label{fig:ten1}
\end{figure}
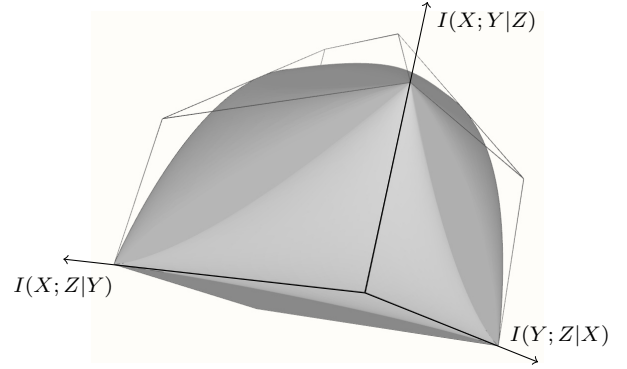
Theorem \ref{thm:basic} summerizes some facts about this region which 
will be used. All facts, except maybe the monotonicty stated in 3) -- for
which we give a direct proof -- are standard \cite{CsK81}. For details 
consult \cite{PP14}.
\begin{theorem}\correcthref\label{thm:basic}
\begin{enumerate}
\item $Z$ can be restricted to have alphabet size $|\mathcal X| |\mathcal
Y|+3$. Consequently $\T(X;Y)$ is compact as a continuous image of a
compact set.
\item $\T(X;Y)$ is convex.
\item If $(X,Y)$ and $(X',Y')$ are independent, then the Minkowsi
(that is, pointwise vector) sum of $\T(X;Y)$ and $\T(X';Y')$
is a subset and a (coordinatewise) lower part of $\T(XX';YY')$.
\item A consequence of {\upshape 2)} and {\upshape 3)} is that
$\T(X^n;Y^n)\supseteq n \T(X;Y)$ and $\T(X^n;Y^n)\ge n\T(X;Y)$,
where $(X^n,Y^n)$ are $n$ {i.i.d.} copies of $(X,Y)$.
\end{enumerate}
\end{theorem}
\begin{proof}
We prove only that the Minkowsi sum is a lower part of $\T(XX';YY')$.
Namely, for every element $p\in
\T(XX';YY')$ one can find points $q\in\T(X;Y)$ and $q'\in\T(X';Y')$ such
that $q+q'\le p$ coordinatewise. Suppose $p$ is generated by $Z$ which is
jointly distributed with $XX'YY'$. The Shannon inequalities
\begin{align*}
\I(XX';YY'|Z) &\ge \I(X;Y|Z) + \I(X';Y'|XYZ), \\
\I(XX';Z|YY') &\ge \I(X;Z|Y) + \I(X';XYZ|Y')-\I(XY;X'Y'), \\
\I(YY';Z|XX') &\ge \I(Y;Z|X) + \I(Y';XYZ|X')-\I(XY;X'Y')
\end{align*}
show that the point $q$ generated by $Z$ and the point $q'$ generated by
$XYZ$ work as required, since $XY$ and $X'Y'$ are independent.
\end{proof}

The G\'acs-K\"orner common information 
$\GK(X;Y)$ can be obtained from the tension region as follows, see
\cite[Corollary 3.3]{PP14}:
\begin{equation}\label{eq:GK}
  \GK(X;Y)=\I(X;Y) - \min\{ r: \langle 0,0,r\rangle \in\T(X;Y) \,\}
\end{equation}
Consequently the G\'acs-K\"orner theorem \cite{gacs-korner} is equivalent to
\begin{theorem}[G\'acs-K\"orner, equivalent form]\correcthref\label{thm:alt}
\begin{enumerate}
\item $\T(X;Y)$
contains only the right endpoint of the interval connecting the origin to
$(0,0,\I(X,Y))$
unless the probability matrix has block structure.
\item $\T(X;Y)$
is separated from the origin (that is, $(0,0,0)$ is not in it) except when,
additionally, $X$ and $Y$ are independent within each such a block.
\end{enumerate}
\end{theorem}

The $\le$ part of the characterization in (\ref{eq:GK}) suffices to deduce
the original form of the G\'acs-K\"orner theorem from Theorem \ref{thm:alt},
and that part follows easily from the claims in Theorem \ref{thm:basic}.
This is discussed in more deatails in Section \ref{sec:MMRV}.

\smallskip

The second ingredient is a special non-Shannon type linear entropy
inequality. The first unconditional entropy inequality which is not a
consequence of the basic Shannon inequalities was discovered by Zhang and
Yeung \cite{ZhY.ineq} in 1998. All such inequalities found so far can be
obtaioned by (possibly iterated) application(s) of the following principle,
see \cite{tarik}.

Let $A$, $B$, and $C$ be disjoint sets of (jointly distributed) random
variables, and $\phi_{AB}$ and $\psi_{BC}$ be linear combinations of entropy
terms containing only variables from $AB$ and $BC$, respectively. Suppose
that for some constant $k$
\begin{equation}\label{eq:princ1}
    \phi_{AB} + \psi_{BC} \ge -k\cdot\I(A;C\|B)
\end{equation}
is a valid entropy inequality (for example, it is a consiequence of basic
Shannon inequalities). Then
\begin{equation}\label{eq:princ2}
   \phi_{AB} + \psi_{BC} \ge 0
\end{equation}
is also a valid entropy inequality.

Indeed, given any distribution on $ABC$, one can use the probabilities
$p'(abc)=p(ab)\cdot p(bc)/p(b)$ to define a new distribution $A'B'C'$ such
that the marginals on $AB$ and $BC$ remain the same, while
$\I(A';C'\|B')=0$. Then $\phi_{AB}=\phi_{A'B'}$, $\psi_{BC}=\psi_{B'C'}$ as
the marginals are the same, thus using inequality (\ref{eq:princ1}) for
$A'B'C'$ one gets
$$
    \phi_{AB}+\psi_{BC}=\phi_{A'B'} + \psi_{B'C'}
     \ge -k\cdot\I(A';C'\|B') = 0,
$$
proving (\ref{eq:princ2}).

The observation mentioned in the abstract comes from inequality
(\ref{eq:princ2}). Fix the random variables in $B$. If there is a
\emph{single} instance of the random variables $A$ jointly distributed with
$B$ such that $\phi_{AB} \le -\eps$, then this fact imposes
$\psi_{BC}\ge\eps$ for \emph{all} possible random variables $C$ jointly
distributed with $B$. This will be used to deduce that under the specified
conditions the origin is separated from the tension region. The relevant
entropy inequality was found by Makarychev et al \cite{MMRV}, and will be
referred to as MMRV. It is
\begin{equation}\label{eq:MMRV}
  \ing_{UVXY}+\Delta_\XYZ\ge 0,
\end{equation}
where $\ing_{UVXY}$ is the Ingleton expression
$$
  -\I(X;Y)+\I(X;Y\|U)+\I(X;Y\|V)+\I(U;V),
$$
and $\Delta_\XYZ$ is
$$
 \I(X;Z\|Y)+\I(Y;Z\|X)+\I(X;Y\|Z).
$$
The MMRV inequality can be verified by the above principle since
$$
  \ing_{UVXY}+\Delta_\XYZ \ge -3\cdot\I(UV;Z\|XY)
$$
is a consequence of the basic Shannon inequalities.

Returning to the second part of Theorem \ref{thm:alt}, we must show that the
origin is not in the tension region $\T(X;Y)$ except when the joint
distribution of $X$ and $Y$ satisfies the stated properties. Observe that
the origin is not in $\T(X;Y)$ if and only if $\Delta_\XYZ>0$ for all
distributions $Z$. According to the MMRV inequality this is the case if
there are distributions $U$ and $V$ such that $\ing_{UVXY}<0$. Thus
exhibiting a single instance of such variables under the given constraints
on $X$ and $Y$ implies immediately the claim. This plan is carried out in
Section \ref{sec:MMRV}..

\def\[{[\mskip -2mu[} \def\]{]\mskip -2mu]}
\def\cd{{\cdot}}

\section{The G\'acs-K\"orner theorem}\label{sec:MMRV}

In this section first we outline how the $\le$ part of the characterization
of the G\'acs-K\"orner information in (\ref{eq:GK}) follows from the claims
of Theorem \ref{thm:basic}. Next we discuss the $\ge$ part, and then proceed
to prove both parts of Theorem \ref{thm:alt}.

By definition, the G\'acs-K\"orner common information $\GK=\GK(X;Y)$ is the
rate of randomness which can be simultaneously extracted from both $X$ and
$Y$ \cite{gacs-korner}. That is, for each $\eps>0$ there is an $n$ and a random variable
$Z$ jointly distributed with $n$ {i.i.d.} copies of $(X,Y)$ such that
$n(\GK-\eps) < H(Z)$, while $Z$ is almost determined by both $X^n$ and
$Y^n$, meaning $\H(Z\|X^n)<n\eps$ and $\H(Z\|Y^n)<n\eps$. These conditions
imply
$$
   \I(X^n;Z|Y^n)<n\eps, ~~ \I(Y^n;Z\|X^n)<n\eps,
$$
and also
\begin{align*}
   \I(X^n;Y^n\|Z) &< \I(X^n;Y^n)+2n\eps -H(Z) <{}\\
                  &< n\big( \I(X;Y)-\GK +3\eps \big).
\end{align*}
This means that for each $\eps>0$ there is an $n$ and a point
$(x_n,y_n,z_n)$ in $\T(X^n;Y^n)$ with $x_n,y_n<n\eps$ and $z_n<
n\big(\I(X;Y)-\GK+3\eps\big)$. Since $\T(X^n;Y^n)\ge n\cdot\T(X;Y)$ it follows
that there is an $(x'_n,y'_n,z'_n) \in \T(X;Y)$ with $x'_n\le x_n/n$,
$y'_n\le y_n/n$, $z'_n\le z_n/n$. As $\T(X;Y)$ is compact,
it also contains a limit of these points as $\eps$ tends to zero.
That is $(0,0,r)\in\T(X;Y)$ for some $r\le \I(X;Y)-\GK$. Consequently
$$
   \min\{r: (0,0,r)\in\T(X;Y)\,\} \le \I(X;Y)-\GK(X;Y),
$$
whis proves the $\le$ part of (\ref{eq:GK}). As it has been remarked before,
it suffices for concluding the G\'acs-K\"orner theorem from Theorem
\ref{thm:alt}. The $\ge$ part is an easy consequence of the following lemma
attributed to Ahlswede and K\"orner and stated explicitly in \cite{tarik}.
For a self-contained proof consult \cite{MMRV}.

\begin{lemma}[Ahlswede-K\"orner \cite{AGK,CsK81}]
Consider $n$ {i.i.d.} copies of $X$, $Y$ and $Z$. There is a random variable $Z'$ 
jointly distributed with $(X^n,Y^n,Z^n)$ such that
\begin{itemize}
\item $\H(Z' \| X^n,Y^n)=0$,
\item $\H(X^n\|Z') = n\H(X\|Z)+o(n)$,
\item $\H(Y^n\|Z') = n\H(Y\|Z)+o(n)$,
\item $\H(X^n,Y^n\|Z') = n\H(X,Y\|Z)+o(n)$. \qed
\end{itemize}
\end{lemma}

For the $\ge$ part of (\ref{eq:GK}) suppose $(0,0,r)\in \T(X;Y)$ is
witnessed by the distribution $Z$. As the first two values in this triplet
are zero, we have
$$
  \I(X;Z)=\I(Y;Z)=\I(X,Y;Z)=\I(X;Y)-r.
$$
Let $Z'$ be the distribution guaranteed by the Ahlswede-K\"orner lemma. An
easy calculation shows that $\H(Z'\|X^n)=o(n)$, $\H(Z'\|Y^n)=o(n)$, and
$$
   \H(Z')=n\I(X,Y;Z)+o(n)= n\big(\I(X;Y)-r\big)+o(n),
$$
providing the rate $\I(X,Y)-r$ which lower bounds $\GK$.

\medskip

Let us turn to the proof of the two statements in Theorem \ref{thm:alt}.
Fix the distributions $X$ and $Y$ on the alphabets $\mathcal
X=\{1,\dots,N\}$ and $\mathcal Y=\{1,\dots, M\}$. Elements of $\mathcal X$
will be denoted by $i$, and elements of $\mathcal Y$ by $j$. Let the probability
of the event $X{=}i, Y{=}j$ be $p_{ij}$. For simplicity we assume that each
element of $\mathcal X$ and each element of $\mathcal Y$ is taken with
non-zero probability.

Vertices of the \emph{block graph} are elements of $\mathcal X \times \mathcal
Y$ with positive probability, and two vertices are connected if they are in
the same row or in the same column. Blocks of the distribution are the
connected components of this graph, and the distribution has no block
structure if this graph is connected.

Now we have all ingredients to prove the first part of Theorem \ref{thm:alt}.
Let $Z$ be a distribution with
$$
   \I(X;Z\|Y)=0, ~~\I(Y;Z\|X)=0, \mbox{ and } \I(X;Y\|Z)=r
$$
for some value $r$, and consider the conditional distributions $Z_{ij} = Z
\|(X{=}i,Y{=}j)$ for $p_{ij}>0$. As $X$ and $Z$ are conditionally
independent assuming $Y$, conditional distributions in the same row (indexed
by elements of $\mathcal Y$) are equal. Similarly, $\I(Y;Z\|X)=0$ implies
that conditional distribution in the same column (indexed by elements of
$\mathcal X$) are also equal. If the block graph is connected, then this implies
that all conditional distributions are the same, meaning that $Z$ is
independent of both $X$ and $Y$. Consequently $r=\I(X;Y\|Z)=\I(X;Y)$ thus
$\GK(X;Y)=0$, as was claimed.

\medskip 
Now we turn to the second part of Theorem \ref{thm:alt}. The exceptional
case requires $X$ and $Y$ to be independent in each connected component of
the block-graph. It means that the non-zero probabilities in the
distribution matrix are split into the disjoint union of combinatorial
rectangles, that is, sets of the form $I \times J$ with $I\subseteq \mathcal
X$ and $J \subseteq \mathcal Y$; moreover in each such a rectangle $X$ and
$Y$ are independent. So assume this is not the case. Then there are
$i_1,i_2\in\mathcal X$ and $j_1,j_2\in\mathcal Y$ such that denoting the
probabilities $p_{i_1j_1}$, $p_{i_1j_2}$, $p_{i_2j_1}$ and $p_{i_2j_2}$ by
$\alpha$, $\beta$, $\gamma$, $\delta$, respectively, one of the following
possibilities hold:
\begin{itemize}
\item[(i)] $\alpha$, $\beta$, $\gamma$ are positive, while $\delta=0$
   (some block is not a combinatorial rectangle),
\item[(ii)] all $\alpha$, $\beta$, $\gamma$, $\delta$  are positive, but 
   $\alpha\delta \neq \beta\gamma$ ($X$ and $Y$ are not independent in 
   some block).
\end{itemize}

By relabeling $\mathcal X$ and $\mathcal Y$ we may assume $i_1=j_1=1$ and
$i_2=j_2=2$, thus $p_{11}=\alpha$, $p_{12}=\beta$, $p_{21}=\gamma$ and
$p_{22}=\delta$. As explained in Section \ref{sec:intro}, to conclude that
the origin is separated from $\T(X;Y)$ it suffices to find random variables $U$ and
$V$ such that $\ing_{UVXY}<0$. To this end choose the probability $p$ close
to $1$, and let $q=1-p$. Define $U$ and $V$ as follows. With probability $p$
$U=X$ and $V=Y$. With probability $q$ $U=\max\{2,X\}$ and
$V=\max\{2,Y\}$.
The probabilities of the joint distribution of $UVXY$ are summarized in
Table \ref{table:dist}, where $\alpha=p_{11}$, $\beta=p_{12}$,
$\gamma=p_{21}$ and $\delta=p_{22}$.

\begin{table}[htb]
\begin{center}
\begin{tabular}{@{}|c@{}c@{}c@{}cl@{\quad}|@{\quad}c@{}c@{}c@{}cl|}
\hline\rule{0pt}{11pt}%
$U$&$ V$&$ X$&$ Y$& Prob &%
$U$&$ V$&$ X$&$ Y$& Prob \\
\hline\rule{0pt}{11pt}%
$1$&$1$&$1$&$1$ & $\alpha\cd p$ & $i$&$1$&$i$&$1$ & $p_{i1}\cd p$\\
$2$&$2$&$1$&$1$ & $\alpha\cd q$  & $i$&$2$&$i$&$1$ & $p_{i1}\cd q$\\
$1$&$2$&$1$&$2$ & $\beta\cd p$  & $1$&$j$&$1$&$j$ & $p_{1j}\cd p$\\
$2$&$2$&$1$&$2$ & $\beta\cd q$ & $2$&$j$&$1$&$j$ & $p_{1j}\cd q$\\
$2$&$1$&$2$&$1$ & $\gamma\cd p$&$i$&$j$&$i$&$j$ & $p_{ij}$\\
$2$&$2$&$2$&$1$ & $\gamma\cd q$&
 \multicolumn{5}{l|}{\hspace{-0.6em}\footnotesize either $i>1$ or $j>1$}\\[2pt]
\hline
\end{tabular}\end{center}
\caption{The distribution $UVXY$}\label{table:dist}
\end{table}

Denote the Ingleton value of this distribution by $\ing(q)$. It is clear
that when $q=0$ (and $p=1$) then $U$ is identical to $X$ and $V$ is
identical to $Y$, thus $\ing(0)=0$. Therefore it suffices to show that the
function $\ing(q)$ is decreasing at $q=0$.

In the marginal distributions of $XU$, $XV$, $YU$, $YV$, $XYU$, $XYV$ and
$UV$ probabilities depending on $p$ or $q$ are summarized in Table
\ref{table:list}. The marginal distributions of $X$ and $Y$ are denoted by
$x_i$ and $y_j$, respectively. In the summary $i$ and $j$ run over values
bigger than 2.

\begin{table}[htb]
\begin{center}\begin{tabular}{|l|l|}
\hline\rule{0pt}{11pt}%
$UX$ & $x_1\cd p$, $x_1\cd q$\\[1.5pt]
$VX$ & $\alpha p$, $\alpha q+\beta$, $\gamma p$, $\gamma q+\delta$,
       $p_{i1}\cd p$, $p_{i1}\cd q+p_{i2}$\\[1.5pt]
$UY$ & $\alpha p$, $\alpha q+\gamma$, $\beta p$, $\beta q+\delta$,
       $p_{1j}\cd p$, $p_{1j}\cd q+p_{2j}$\\[1.5pt]
$VY$ & $y_1\cd p$, $y_1\cd q$\\[3pt]
$UXY$ & $\alpha p$, $\alpha q$, $\beta p$, $\beta q$, $p_{1j}\cd p$,
       $p_{1j}\cd q$\\[1.5pt]
$VXY$ & $\alpha p$, $\alpha q$, $\gamma p$, $\gamma q$, $p_{i1}\cd p$,
       $p_{i1}\cd q$\\[1.5pt]
$UV$ & $\alpha p$, $\beta p$, $\gamma p$,
       $\alpha q+\beta q+\gamma q+\delta$, \\[1pt]
     & ~~~ $p_{i1}\cd p$, $p_{i1}\cd q+p_{i2}$,$p_{1j}\cd p$, 
       $p_{1j}\cd q+p_{2j}$\\[2pt]
\hline
\end{tabular}\end{center}
\caption{Marginals of $UVXY$ depending on $p$ or $q$}\label{table:list}
\end{table}

Let $\[x\]=h(x)=-x\log x$ where $\log$ is the natural logarithm. The
entropy, up to the scaling factor $\log 2$, is the sum of the $h$-value of
the corresponding probabilities. Probabilities occuring in the Ingleton
expression $\ing(q)$ which depend on $p$ or $q$ are listed in Table
\ref{table:list}. The $h$-value of the probabilities in the first four lines
are added, and those in the last three lines are subtracted. Multiplications
inside $\[xy\]$ can be eliminated using the identity $\[xy\] =
x\cdot\[y\]+y\cdot\[x\]$, for example, $\[x\cd p\]+\[x\cd q\] = \[x\] +
x\cd(\[p\]+\[q\])$. Using these rules part of $\ing(q)$ which depends on $p$
and $q$ simplifies to
\begin{align}\label{eq1}
    & \[\alpha - \alpha q\] + \[\beta+\alpha q\]+\[\gamma+\alpha q\] +{}\\
    & +\[\delta+\beta q\] + \[\delta+\gamma q\] -  \[\delta+\alpha q+
      \beta q+\gamma q \] \nonumber.
\end{align}
For $q=0$ this espression equals $\[\alpha\]+\[\beta\]+\[\gamma\]+\[\delta\]$,
and we want to show that it strictly decreases for small positive values of $q$.
For positive $x$ we have
$$
   \[x+\eps\] = \[x\]-\eps-\eps\log x + O(\eps^2).
$$
If
$\alpha$, $\beta$, $\gamma$ and $\delta$ are all positive, then
(\ref{eq1}) rewrites to
$$
 \[\alpha\]+\[\beta\]+\[\gamma\]+ \[\delta\] +\alpha q \log\frac
 {\alpha\delta}{\beta\gamma} + O(q^2).
$$
The clearly decreases for small value of $q$ when $\alpha\delta<
\beta\gamma$, which can be clearly achieved by optionally swapping 1 and 2 in $\mathcal
X$.

In the other considered case $\alpha$, $\beta$ and $\gamma$ are positive
and $\delta=0$, when (\ref{eq1}) rewrites to
$$
 \[\alpha - \alpha q\] + \[\beta+\alpha q\]+\[\gamma+\alpha q\] +
 +q\[\beta\]+q\[\gamma\]-q\[\alpha+\beta+\gamma\]-\alpha\[q\].
$$
It is clearly decreasing for small positive values of
$q$ as $\[x\]$ has plus infinite derivative at zero, and finite derivative
everywhere else. This proves the G\'acs-K\"orner theorem.


\section*{Acknowledgment}
The research reported in this paper has been partially supperted by the
ERC Advanced Grant ERMiD,

\end{document}